\documentclass[reqno,12pt,a4paper]{amsart}

\usepackage{enumerate}
\usepackage{amstext,amsmath,amsthm,amsfonts,amssymb,amscd}
\usepackage{latexsym,mathrsfs,dsfont,euscript}
\usepackage{mathtools}
\usepackage{fullpage}
\usepackage{txfonts}
\usepackage{color}

\usepackage{hyperref} 
\hypersetup{
  colorlinks,%
  citecolor=blue,%
    filecolor=red,%
    linkcolor=red,%
    urlcolor=blue
}

\usepackage{color}

\usepackage{tikz}
\usepackage{pgfplots}
\usetikzlibrary{arrows,calc}
\usetikzlibrary{patterns}
\usetikzlibrary{plotmarks}
\usetikzlibrary{calc,through,backgrounds}
\usetikzlibrary{calendar,snakes}
\usetikzlibrary{positioning}

\newtheorem{theorem}{Theorem}
\newtheorem{proposition}[theorem]{Proposition}
\newtheorem{lemma}[theorem]{Lemma}

\theoremstyle{definition}

\theoremstyle{remark}

\numberwithin{equation}{section}

\newcommand{\abs}[1]{\left\vert#1\right\vert}

\newcommand{\proin}[2]{\left<#1,#2\right>}
\newcommand{\norm}[1]{\left\Vert#1\right\Vert}

\allowdisplaybreaks
\begin{document}
\title[]{Fractional uncertainty}
%
%


\author[]{Hugo Aimar}
\email{haimar@santafe-conicet.gov.ar}
\author[]{Pablo Bolcatto}
\email{pablo.bolcatto@santafe-conicet.gov.ar}
\author[]{Ivana G\'{o}mez}
\email{ivanagomez@santafe-conicet.gov.ar}
\thanks{This work was supported by the CONICET, ANPCyT-MINCyT and UNL}
%
%
%

\begin{abstract}
We use techniques of dyadic analysis in order to prove that, for every $0<s<\tfrac{1}{2}$, there exists a positive constant $\gamma(s)$ such that the inequality
$$\left(\iint_{\mathbb{R}^2}\abs{x-y}^{2s-1}\abs{\varphi(x)}\abs{\varphi(y)}dx dy\right)\left(\iint_{\mathbb{R}^2}\abs{x-y}^{-2s-1}\abs{\varphi(x)-\varphi(y)}^2 dx dy\right)\geq\gamma(s)$$ holds for every $\varphi$ with $\norm{\varphi}_{L^2(\mathbb{R})}=1$. The second integral on the left hand side is the energy quadratic form of order $s$,  which for the limit case $s=1$ gives the local form $Var\abs{\hat{\varphi}}^2$ or $\int\abs{\nabla\varphi}^2$. The first is a natural substitution of the position form, which on the Haar system shows the same behavior of the classical $Var\abs{\varphi}^2$.
\end{abstract}
\maketitle

\section{Introduction}

Let us denote, as usual, by $\psi=\psi(x,t)$ the wave function describing the non relativistic state of a quantum system at time $t$. The evolution in time of the system described by $\psi$ is governed by the Schr\"{o}dinger equation
\begin{equation*}
-\frac{\hbar}{i}\frac{\partial \psi}{\partial t}= H \psi,
\end{equation*}
with $H$ the Hamiltonian operator determined by the system and $\hbar=\tfrac{h}{2\pi}$ where $h$ is the Planck constant. The Hamiltonian operator $H$ takes into account the total energy of the system. The canonical quantization of the classical description of the energy in terms of the kinetic and potential energies, $e= \frac{\vec{p}\cdot\vec{p}}{2 m}+V(x)$, where we substitute the momentum $\vec{p}$ by the vector operator $-i\hbar\nabla$, gives the most classical form of the Schr\"{o}dinger equation for a particle of mass $m$,
\begin{equation*}
-\frac{\hbar}{i}\frac{\partial \psi}{\partial t}=\frac{\hbar^2}{2m}\Delta\psi + V\psi,
\end{equation*}
where the kinetic energy is essentially the Laplace operator $\Delta\psi=\nabla^2\psi=\sum_{i=1}^{3}\frac{\partial^2\psi}{\partial x_i^2}$. When $V=0$ we say that we have a free system. The Laplace operator, in the calculus of variations, is the Euler-Lagrange operator corresponding to the energy quadratic form
$\mathcal{E}_1(\varphi)=\int\abs{\nabla\varphi}^2dx$
associated to the energy bilinear form
\begin{equation*}
E_1(\varphi,\eta)=\int\nabla\varphi\nabla\overline{\eta} dx.
\end{equation*}
In other words, the harmonic functions $u$ ($\Delta u=0$) are the minimizers of $\mathcal{E}_1$ on adequate function spaces.

In the classical book of Landkoff \cite{LandkofBook}, a huge family of energy bilinear and quadratic forms is introduced and the corresponding potential theory is thoroughly developed. In particular the family $\mathcal{E}_s$ ($0<s<1$) from which $\mathcal{E}_1$ can be regarded as a limit case. Precisely, the energy bilinear form on $\mathbb{R}^n$
\begin{equation*}
E_s(\varphi,\eta)=\iint_{\mathbb{R}^n\times \mathbb{R}^n}\left[\frac{\varphi(x)-\varphi(y)}{\abs{x-y}^s}\right]\left[\frac{\overline{\eta}(x)-\overline{\eta}(y)}{\abs{x-y}^s}\right]\frac{dx dy}{\abs{x-y}^n}.
\end{equation*}
and the associated energy quadratic form
\begin{equation*}
\mathcal{E}_s(\varphi)=E_s(\varphi,\varphi)=\iint_{\mathbb{R}^n\times \mathbb{R}^n}\abs{\frac{\varphi(x)-\varphi(y)}{\abs{x-y}^s}}^2\frac{dx dy}{\abs{x-y}^n}.
\end{equation*}
The Euler-Lagrange operator associated to $\mathcal{E}_s$ is the fractional Laplacian $(-\Delta)^s$, which at least in the principal value sense, is given by
\begin{equation}\label{eq:fractionallaplacianRn}
(-\Delta)^s\varphi(x)=\int_{y\in \mathbb{R}^n}\frac{\varphi(x)-\varphi(y)}{\abs{x-y}^{n+s}}dy.
\end{equation}
Schr\"{o}dinger equations with Hamiltonians involving fractional Laplacians have been considered before, see for example \cite{Laskin00}, \cite{Laskin02}, \cite{Wei16}, \cite{Laskin16}. These authors introduce and discuss uncertainty relations of the setting. We aim to use the energy bilinear forms approach described before, in order to prove some uncertainty relations that sound natural to the context of finite $s$-energy forms for small $s$. To precise the above considerations, let us briefly review the uncertainty inequality for the Heisenberg Principle. At this point we have to mention that there is an extensive literature on uncertainty inequalities including Balian-Low's theorem and many other extensions and point of view. We only refer here to the classical survey \cite{FoSit97} and to \cite{BePo06}. Let us also refer to \cite{OkoStri05} and \cite{OkoSalTep08} for related results on fractals and some metric measure spaces.

For the sake of simplicity we shall remain in one dimension space. Set $\widehat{\varphi}$ to denote the Fourier transform of $\varphi\in L^2(\mathbb{R})$. The uncertainty inequality reads
\begin{equation*}
Var\abs{\varphi}^2\,Var\abs{\widehat{\varphi}}^2\geq \frac{1}{16\pi^2}
\end{equation*}
for every $\varphi\in L^2(\mathbb{R})$ with $\norm{\varphi}_2=1$. Here
\begin{equation*}
Var\abs{\varphi}^2=\inf_{a\in\mathbb{R}}\left(\int_{\mathbb{R}}(x-a)^2\abs{\varphi(x)}^2dx\right)
\end{equation*}
and similarly
\begin{equation*}
Var\abs{\widehat{\varphi}}^2=\inf_{\alpha\in\mathbb{R}}\left(\int_{\mathbb{R}}(\xi-\alpha)^2\abs{\widehat{\varphi}(\xi)}^2d\xi\right).
\end{equation*}
Since the term $Var\abs{\widehat{\varphi}}^2$ can be rewritten as a constant times $\int_{\mathbb{R}}\abs{\frac{d\varphi}{dx}}^2 dx=\int_{\mathbb{R}}\abs{\nabla\varphi}^{2}=\mathcal{E}_1(\varphi)$, the energy quadratic form for $\varphi$ or in physical terms $(\triangle \vec{p})^2$, we aim to search for an adequate functional $\mathcal{Q}_s$, substituting the classical $(\triangle \vec{q})^2$, defined on the wave functions $\varphi$, with $\norm{\varphi}_2=1$, in order to obtain an uncertainty relation of the type
\begin{equation*} 
\mathcal{Q}_s(\varphi)\mathcal{E}_s(\varphi)\geq c
\end{equation*}
for a positive constant which could depend on $s$ but not on $\varphi$.

At this point we have to mention the work of two of the authors in \cite{AiBoGo13} and \cite{AcAiBoGo16}, to point out that the Haar wavelet system is the sequence of eigenfunctions of a special nonlocal fractional differential operator which resemble \eqref{eq:fractionallaplacianRn} with $n=1$. Moreover, the Schr\"{o}dinger equation for a free particle with these type of kinetic energies and prescribed initial condition are explicitly solved in \cite{AiBoGo13} through the Haar system. In Section~\ref{sec:dyadicfractionaluncertainty} we prove a precise dyadic uncertainty relation with a position functional that also, as the energy form, comes from a bilinear positive definite form. Once these results are established we go back, in Section~\ref{sec:euclideanfractionaluncertainty}, to the Euclidean world in order to produce the desired uncertainty relation for $s$ small.

\section{The dyadic fractional uncertainty inequality}\label{sec:dyadicfractionaluncertainty}
Let $\mathcal{D}$ be the countable family of all dyadic intervals in $\mathbb{R}^+$. An interval $I$ belongs to $\mathcal{D}$ if and only if there exists $j\in\mathbb{Z}$ and $k\in\mathbb{Z}$ such that $I=I^j_k=(k2^{-j},(k+1)2^{-j}]$ with $\mathcal{D}_j$ we denote all the dyadic intervals of level $j$, i.e. $\mathcal{D}_j=\{I^j_k:\ k\in\mathbb{Z}\}$. The Haar system associated to $\mathcal{D}$ is denoted by $\mathcal{H}$. There is a one to one  correspondence between $\mathcal{D}$ and $\mathcal{H}$. If $I=I^j_k$ then $h_I(x)=2^{j/2}h(2^jx-k)$, with $h(x)=\mathcal{X}_{(0,1/2]}-\mathcal{X}_{(1/2,1]}$. If $h=h^j_k$ we write $I(h)$ to denote the interval $I^j_k$. The system $\mathcal{H}$ is an orthonormal basis for $L^2(\mathbb{R}^+)$. A natural metric in $\mathbb{R}^+$ is provided by the dyadic intervals.

\begin{lemma}[\cite{AiBoGo13}]
Let $\mathcal{D}$ be the family of dyadic intervals in $\mathbb{R}^+$. The function
\begin{equation*}
\delta(x,y)=\inf \{\abs{I}: I\in\mathcal{D} \textrm{ and } x,y\in I\}
\end{equation*}
is a distance on $\mathbb{R}^+$. For every $x$ and $y$ we have the inequality $\abs{x-y}\leq\delta(x,y)$.
\end{lemma}

Of course $\delta(x,y)$ in not equivalent pointwise to $\abs{x-y}$ but globally they have the same integrability properties.
\begin{lemma}
With the above notation and $\alpha>0$ we have,
\begin{enumerate}[(a)]
	\item the $\delta$-ball with center $x\in\mathbb{R}^+$ and radius $r>0$ is the largest dyadic interval containing $x$ with measure less than $r$. Let us denote this interval by $B_{\delta}(x,r)$ or by $I(x,r)$;
	\item $\int_{B_\delta(x,r)}\frac{1}{\delta^{1-\alpha}}dy=\int_{I(x,r)}\frac{1}{\delta^{1-\alpha}}dy=\tfrac{1}{2(1-2^{-\alpha})}\abs{I(x,r)}^{\alpha}\simeq r^\alpha$;
	\item $\int_{B_\delta(x,r)}\frac{1}{\delta^{1+\alpha}}dy=+\infty$;
	\item $\int_{y\notin B_\delta(x,r)}\frac{1}{\delta^{1+\alpha}}dy = \tfrac{2^{-\alpha}}{2(1-2^{-\alpha})}\abs{I(x,r)}^{-\alpha}\simeq r^{-\alpha}$;
	\item $\int_{y\notin B_\delta(x,r)}\frac{1}{\delta^{1-\alpha}}dy=+\infty$;
	\item $\int_{y\in B_\delta(x,r)}\frac{1}{\delta(x,y)}dy=\int_{y\notin B_\delta(x,r)}\frac{1}{\delta(x,y)}dy=+\infty$.
\end{enumerate}
\end{lemma}
\begin{proof}
	We shall only check (b) and (d) where the precise constants have to be determined. For (c), (e) and (f) we only need to compare each integral with a divergent sequence.

(b)
	\begin{align*}
	\int_{B_\delta(x,r)}\frac{dy}{\delta^{1-\alpha}(x,y)} &= \sum_{k=0}^{\infty}
	\int_{\{y: \delta(x,y)=2^{-k}\abs{I(x,r)}\}}\frac{dy}{\delta^{1-\alpha}(x,y)}\\
	&= \sum_{k=0}^{\infty} 2^{k(1-\alpha)}\abs{I(x,r)}^{-1+\alpha}\cdot\abs{\{y: \delta(x,y)=2^{-k}\abs{I(x,r)}\}}\\
	&= \frac{1}{2}\cdot\frac{1}{1-2^{-\alpha}}\abs{I(x,r)}^{\alpha}.
	\end{align*}
	
(d)	
	\begin{align*}
\int_{y\notin I(x,r)}\frac{dy}{\delta^{1+\alpha}(x,y)}
&= \sum_{k=1}^{\infty}\int_{\{y:\delta(x,y)=2^{k}\abs{I(x,r)}\}}\frac{dy}{\delta^{1+\alpha}(x,y)}\\
&= \sum_{k=1}^{\infty} 2^{-k(1+\alpha)}\abs{I(x,r)}^{-\alpha-1}\cdot\abs{\{y: \delta(x,y)=2^{k}\abs{I(x,r)}\}}\\
&= \sum_{k=1}^{\infty} 2^{-k(1+\alpha)}\abs{I(x,r)}^{-\alpha-1}\frac{2^k\abs{I(x,r)}}{2}\\
&= \frac{1}{2}\cdot\frac{2^{-\alpha}}{1-2^{-\alpha}}\abs{I(x,r)}^{-\alpha}.
\end{align*}
\end{proof}
Let us introduce two bilinear forms acting on functions defined on the positive real numbers $\mathbb{R}^+$ in terms of the metric $\delta(x,y)$. First let us introduce the energy bilinear form associated to the metric $\delta$ on $\mathbb{R}^+$ of order $0<s<\tfrac{1}{2}$
\begin{equation*}
E^\delta_s(\varphi,\psi)=\iint_{(\mathbb{R}^+)^2}\left[\frac{\varphi(x)-\varphi(y)}{\delta^s(x,y)}\right]
\left[\frac{\overline{\psi}(x)-\overline{\psi}(y)}{\delta^s(x,y)}\right]\frac{dx dy}{\delta(x,y)}.
\end{equation*}
The energy form $\mathcal{E}^\delta_s$ is given by the value of $E^\delta_s$ on the diagonal. Precisely
\begin{equation*}
\mathcal{E}^\delta_s(\varphi) := E^\delta_s(\varphi,\varphi)=
\iint_{(\mathbb{R}^+)^2}\abs{\frac{\varphi(x)-\varphi(y)}{\delta^s(x,y)}}^2\frac{dx dy}{\delta(x,y)}.
\end{equation*}
The second bilinear form that we introduce, and we call the \textit{position} bilinear form in the dyadic setting of $\mathbb{R}^+$ is given by
\begin{equation*}
Q^\delta_s(\varphi,\psi)=
\iint_{(\mathbb{R}^+)^2}[\delta^s(x,y)\varphi(x)][\delta^s(x,y)\overline{\psi}(y)]\frac{dx dy}{\delta(x,y)}.
\end{equation*}
The value of $Q^\delta_s$ on the diagonal $\varphi=\psi$ provides the \textit{position} quadratic form
\begin{equation*}
\mathcal{Q}^{\delta}_s(\varphi) := Q^\delta_s(\varphi,\varphi)=
\iint_{(\mathbb{R}^+)^2}[\delta^{2s}(x,y)\varphi(x)\overline{\varphi}(y)]\frac{dx dy}{\delta(x,y)}.
\end{equation*}
Regarding the finiteness of the quadratic forms $\mathcal{E}^\delta_s(\varphi)$ and $\mathcal{Q}^\delta_s(\varphi)$ let us notice that while $\mathcal{Q}^\delta_s(\varphi)$ requires only some boundedness or integrability for $\varphi$, the finiteness of the energy form instead involves some regularity of $\varphi$. Since we want to compute $\mathcal{E}^\delta_s(h)$ for $h\in\mathcal{H}$ and the Haar functions are discontinuous, it seems at first that there is no hope for the finiteness of $\mathcal{E}^\delta_s(h)$. Nevertheless, every Haar function is of Lipschitz (H\"{o}lder $1$) class with respect to $\delta$. Actually this is true for functions built as linear combinations of indicators of dyadic intervals. See \cite{AiGoPetermichl} for details regarding these basic facts of dyadic analysis.

In order to have an insight of the action of the quadratic form $\mathcal{Q}^\delta_s$, let us compute both quantities $Var\abs{h}^2$ and $\mathcal{Q}^\delta_s(h)$ for every $h\in\mathcal{H}$.
\begin{lemma}\label{lem:VarhQsform}
	For every $h\in\mathcal{H}$ and $\gamma_1(s)= \frac{1}{2}\left(\frac{2^{1-2s}-1}{1-2^{-2s}}\right)$ we have
	\begin{enumerate}[(a)]
		\item $Var \abs{h}^2=\tfrac{1}{12}\abs{I(h)}^2$, and
		\item $\mathcal{Q}^\delta_s(h)=\gamma_1(s)\abs{I(h)}^{2s}$.
	\end{enumerate}
\end{lemma}
\begin{proof}
	(a) Take $h\in\mathcal{H}$. Set $I(h)$ to denote the dyadic support of $h$. Let $m(h)$ be the center of $I(h)$. Since $\abs{h}^2=\frac{\mathcal{X}_{I(h)}}{\abs{I(h)}}$ we have that the mean value $\int_{\mathbb{R}}x\abs{h}^2dx=\tfrac{1}{\abs{I(h)}}\int_{I(h)}xdx=m(h)$. Then
\begin{equation*}
Var\abs{h}^2=\int_{\mathbb{R}}(x-m(h))^2\abs{h(x)}^2dx=
\frac{1}{\abs{I(h)}}\int_{I(h)}(x-m(h))^2dx=\frac{1}{12}\abs{I(h)}^2.
\end{equation*}

(b) In the sequel we shall write $I_+$ and $I_-$ to denote the right and left halves of $I$.
\begin{align*}
\mathcal{Q}^\delta_s(h)&=\iint_{(\mathbb{R}^+)^2}\delta^{2s}(x,y)h(x)h(y) \frac{dx dy}{\delta(x,y)}\\
&=\int_{I(h)}\int_{I(h)}\delta^{2s}(x,y) h(x)h(y) dx dy\\
&=\int_{I_-(h)}\int_{I_-(h)}+\int_{I_+(h)}\int_{I_+(h)}+\int_{I_-(h)}\int_{I_+(h)}+\int_{I_+(h)}\int_{I_-(h)}\\
&= I_1+I_2+I_3+I_4.
\end{align*}

\begin{figure}[htp]
	\begin{center}
		\begin{tikzpicture}[scale=1.5]
		%
		\draw[line width=.5pt, color=gray] (0,0)--(2.5,0);
		\draw[line width=.5pt, color=gray] (0,0)--(0,2.5);
		\draw[line width=.5pt, color=black!80] (1,1)--(1,2);
		\draw[line width=.5pt, color=black!80] (1,1)--(2,1);
		\draw[line width=.5pt, color=black!80] (1.5,1)--(1.5,2);
		\draw[line width=.5pt, color=black!80] (1,1.5)--(2,1.5);
		\draw[line width=.5pt, color=black!80] (2,1)--(2,2);
		\draw[line width=.5pt, color=black!80] (1,2)--(2,2);
		\draw[line width=.5pt, color=gray] (0,0)--(2.5,2.5); 
		\node at (1.25,1.25) {\scalebox{1}{$+$}};
		\node at (1.75,1.75) {\scalebox{1}{$+$}};		
		\node at (1.25,1.75) {\scalebox{1}{$-$}};
		\node at (1.75,1.25) {\scalebox{1}{$-$}};
		\node at (1.4,1.1) {\scalebox{1}{$I_1$}};
		\node at (1.9,1.6) {\scalebox{1}{$I_2$}};		
		\node at (1.4,1.6) {\scalebox{1}{$I_3$}};
		\node at (1.9,1.1) {\scalebox{1}{$I_4$}};			
		\draw[line width=2pt, color=black!90] (1,0)--(2,0);
		\node at (1.5,-.3) {\scalebox{1}{$I(h)$}};
		\draw[line width=2pt, color=black!90] (0,1)--(0,2);
		\node at (-.3,1.5) {\scalebox{1}{$I(h)$}};
		\end{tikzpicture}
	\end{center}
	\caption{The sign of $h(x)h(y)$}
\end{figure}
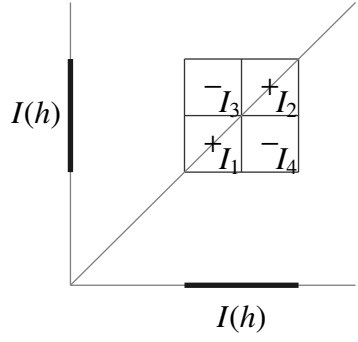

\medskip
It is easy to check that $I_1=I_2$ and $I_3=I_4$. Nevertheless, while $\delta$ is constant in $(I_-(h)\times I_+(h))\cup (I_+(h)\times I_-(h))$, this is not the case in $(I_-(h)\times I_-(h))\cup(I_+(h)\times I_+(h))$. We shall compute $I_1$ and $I_3$. For $I_1$ we decompose the integral in the level sets of $\delta$,
\begin{align*}
I_1
&=\int_{I_{-}(h)}\int_{I_{-}(h)}\delta^{2s-1}(x,y)h(x)h(y) dx dy\\
&=\frac{1}{\abs{I(h)}}\int_{I_{-}(h)}\int_{I_{-}(h)}\delta^{2s-1}(x,y)dx dy\\
&=\frac{1}{\abs{I(h)}}\sum_{j=0}^{\infty}\iint\limits_{(I_-(h)\times I_-(h))\cap \{(x,y):\delta(x,y)=2^{-j}\abs{I_-(h)}\}}\delta^{2s-1}(x,y)dx dy\\
&=\frac{1}{\abs{I(h)}}\sum_{j=0}^{\infty}\abs{I_-(h)}^{2s-1}2^{j(1-2s)}\abs{(I_-(h)\times I_-(h))\cap \{(x,y):\delta(x,y)=2^{-j}\abs{I_-(h)}\}}\\
&=\frac{2\abs{I_-(h)}^{2s-1}}{\abs{I(h)}}\abs{I_-(h)}^2\sum_{j=0}^{\infty}2^{j(1-2s)}2^j\frac{1}{2^{2(j+1)}}\\
&= 2\frac{\abs{I_-(h)}^{2s}}{\abs{I(h)}}\abs{I_-(h)}\sum_{j=0}^{\infty}\frac{1}{4}2^{-2sj}\\
&=\frac{1}{4}2^{-2s}\frac{1}{1-2^{-2s}}\abs{I(h)}^{2s}.
\end{align*}
Let us now compute $I_3$. On the product $I_-(h)\times I_+(h)$ the distance $\delta$ is constant and takes the value $\delta(x,y)=\abs{I(h)}$. Hence
\begin{align*}
I_3 &= \int_{I_-(h)}\int_{I_+(h)} \delta^{2s-1}(x,y) h(x)h(y) dx dy\\&=-\frac{1}{\abs{I(h)}}\abs{I(h)}^{2s-1}\abs{I_-(h)\times I_+(h)}\\
&=-\frac{1}{4}\abs{I(h)}^{2s}.
\end{align*}
So that
\begin{equation*}
\mathcal{Q}^\delta_s(h) =2 I_1+ 2 I_3
=2\frac{1}{4}\left(\frac{2^{-2s}}{1-2^{-2s}}-1\right)\abs{I(h)}^{2s}
=\frac{1}{2}\left(\frac{2^{1-2s}-1}{1-2^{-2s}}\right)\abs{I(h)}^{2s}.
\end{equation*}
Notice that the constant $\gamma_1(s)=\frac{1}{2}\left(\frac{2^{1-2s}-1}{1-2^{-2s}}\right)$ is positive if $0<s<\tfrac{1}{2}$.
\end{proof}

The above result shows that except for the constant, there is a formal continuity because with $s=1$, (b) look like (a) on the Haar system $\mathcal{H}$.

For the dyadic energy form of order $s$ acting on the Haar system we still have a simple behavior that suggest a possible uncertainty relation.
\begin{lemma}\label{lem:energyvalueH}
	For every $h\in\mathcal{H}$ we have
	\begin{equation*}
	\mathcal{E}^\delta_s(h)=\gamma_2(s)\abs{I(h)}^{-2s}
	\end{equation*}
with $\gamma_2(s)=\tfrac{2-2^{-2s}}{1-2^{-2s}}$.
\end{lemma}
Before proving the lemma, let us observe that $Var\abs{\hat{h}}^2=+\infty$ and that from (b) of Lemma~\ref{lem:VarhQsform} and the above statement $\mathcal{Q}^\delta_s(h)\cdot\mathcal{E}^\delta_s(h)=\gamma(s)$ for every $h\in\mathcal{H}$ and $\gamma(s)=\gamma_1(s)\cdot \gamma_2(s)$.

We shall prove that this inequality holds for general $f$ in the unit ball of $L^2$, not just for the Haar system.

\begin{proof}[Proof of Lemma~\ref{lem:energyvalueH}]
We shall keep using the notation that we used in the proof of the above Lemma. Take $h\in\mathcal{H}$, then
\begin{align*}
\mathcal{E}^\delta_s(h) &= \iint_{(\mathbb{R}^+)^2}\frac{\abs{h(x)-h(y)}^2}{\delta^{1+2s}(x,y)}dx dy\\
	&= \iint_{I(h)\times I(h)}+\iint_{I(h)\times I^{c}(h)}+\iint_{I^{c}(h)\times I(h)}\\
	&= I_1+I_2+I_3,
\end{align*}
where $I^{c}(h)$ denotes the complement $\mathbb{R}\setminus I(h)$ of $I(h)$.
\begin{figure}[htp]
	\begin{center}
		\begin{tikzpicture}[scale=1.3]
		%
		\draw[line width=.5pt, color=gray] (0,0)--(2,0)--(2,2)--(0,2)--(0,0);		
		\draw[line width=.5pt, color=gray] (0,0)--(2,2); 
		\draw[line width=.8pt, color= black!80, fill= lightgray!80] (0,1)--(1,1)--(1,1.5)--(0,1.5)--(0,1);
		\draw[line width=.8pt, color= black!80, fill= lightgray!80] (1.5,1)--(2,1)--(2,1.5)--(1.5,1.5)--(1.5,1);
		\draw[line width=.8pt, color= black!80, fill= black!80] (1,1)--(1.5,1)--(1.5,1.5)--(1,1.5)--(1,1);
		\draw[line width=.8pt, color= black!80, fill= black!50] (1,1.5)--(1.5,1.5)--(1.5,2)--(1,2)--(1,1.5);
		\draw[line width=.8pt, color= black!80, fill= black!50] (1,1)--(1,0)--(1.5,0)--(1.5,1)--(1,1);
		\node at (1.2,-.3) {\scalebox{1}{$I(h)$}};
		\node at (-.4,1.2) {\scalebox{1}{$I(h)$}};
		\end{tikzpicture}
	\end{center}
	\caption{The three integration regions}
\end{figure}
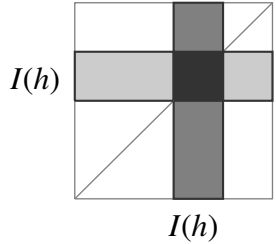

\noindent Notice first that with $b(h)$ the right end point of $I(h)$,
\begin{align*}
I_2=I_3&=\int_{x\notin I(h)}\left(\int_{y\in I(h)}\frac{(h(x)-h(y))^2}{\delta^{1+2s}(x,y)} dy\right) dx\\
&=\int_{x\notin I(h)}\left(\int_{y\in I(h)}\frac{h^2(y)}{\delta^{1+2s}(x,y)} dy\right) dx\\
&= \int_{x\notin I(h)}\frac{1}{\delta^{1+2s}(x,b(h))}\norm{h}^2_2 dx\\
&=\int_{x\notin I(h)}\frac{dx}{\delta^{1+2s}(x,b(h))}\\
&=\sum_{j=1}^{\infty}\int_{\left\{x: \delta(x,b(h))=2^j\abs{I(h)}\right\}}(2^j\abs{I(h)})^{-1-2s}dx\\
&=\sum_{j=1}^{\infty}2^{-j(1+2s)}\abs{I(h)}^{-1-2s}2^{j-1}\abs{I(h)}\\
&=\abs{I(h)}^{-2s}\frac{2^{-2s}}{2(1-2^{-2s})}.
\end{align*}

To compute $I_1$, notice that $h(x)=h(y)$ on $(I_-(h)\times I_-(h))\cup(I_+(h)\times I_+(h))$, hence
\begin{align*}
I_1&= \iint_{I(h)\times I(h)}\frac{(h(x)-h(y))^2}{\delta^{1+2s}(x,y)} dx dy\\
&=\abs{I(h)}^{-1-2s}\iint_{I(h)\times I(h)}(h(x)-h(y))^2 dx dy\\
&=\abs{I(h)}^{-1-2s}\int_{I(h)}\int_{I(h)}\left[h^2(x)+h^2(y)-2h(x)h(y)\right] dy dx\\
&=\abs{I(h)}^{-1-2s}\left(\int_{I(h)}h^2(x) dx\int_{I(h)}dy+\int_{I(h)}h^2(y) dy\int_{I(h)}dx\right)\\
&= 2\abs{I(h)}^{-2s}.
\end{align*}
So that
\begin{equation*}
\mathcal{E}^\delta_s (h)= I_1+2I_2=\left(2+\frac{2^{-2s}}{1-2^{-2s}}\right)\abs{I}^{-2s}.
\end{equation*}
\end{proof}
Let $\mathcal{S}(\mathcal{H})$ be the linear span generated by $\mathcal{H}$ in $L^2$. In other words $\mathcal{S}(\mathcal{H})$ is the set of all (finite) linear combinations of members of $\mathcal{H}$. Let us state the dyadic uncertainty inequality that we are in position to prove with the above lemmas and some orthogonality properties that we prove after the statement of the theorem.
\begin{theorem}\label{thm:inequalityQandE}
Let $0<s<\tfrac{1}{2}$ be given and $\gamma(s)=\gamma_1(s)\cdot \gamma_2(s)$. Then, the inequality
\begin{equation*}
\mathcal{Q}^\delta_s(\varphi)\cdot\mathcal{E}^\delta_s(\varphi)\geq \gamma(s)\norm{\varphi}^4_2
\end{equation*}
holds for every $\varphi\in\mathcal{S}(\mathcal{H})$.
\end{theorem}
The theorem above is an easy consequence of the following and the previous lemmas.
\begin{lemma}\label{lemm:valueEandQinpsi}
Let $0<s<\tfrac{1}{2}$.
\begin{enumerate}[(1)]
\item $E^\delta_s(h,\widetilde{h})=0$ for $h\neq \widetilde{h}$ both in $\mathcal{H}$;
\item $Q^\delta_s(h,\widetilde{h})=0$ for $h\neq \widetilde{h}$ both in $\mathcal{H}$;
\item $\mathcal{E}^\delta_s(\varphi)= \gamma_2(s)\sum_{h\in\mathcal{H}}\frac{\abs{\proin{\varphi}{h}}^2}{\abs{I(h)}^{2s}}$ for every $\varphi\in\mathcal{S}(\mathcal{H})$;
\item $\mathcal{Q}^\delta_s(\varphi)= \gamma_1(s) \sum_{h\in\mathcal{H}}\abs{I(h)}^{2s}\abs{\proin{\varphi}{h}}^2$ for every $\varphi\in\mathcal{S}(\mathcal{H})$.
\end{enumerate}
\end{lemma}
\begin{proof}[Proof of the Theorem~\ref{thm:inequalityQandE}]
Let $\varphi=\sum_{h\in\mathcal{H}}\proin{\varphi}{h}h$ be a given function in $L^2(\mathbb{R}^+)$ such that $\proin{\varphi}{h}=0$ except for a finite subset of $\mathcal{H}$. Then, from Schwartz inequalities for series
\begin{align*}
\norm{\varphi}^2_2 &=\sum_{h\in\mathcal{H}}\abs{\proin{\varphi}{h}}^2\\
&=\sum_{h\in\mathcal{H}}\abs{I(h)}^s\abs{\proin{\varphi}{h}}\abs{I(h)}^{-s}\abs{\proin{\varphi}{h}}\\
&\leq\left(\sum_{h\in\mathcal{H}}\abs{I(h)}^{2s}\abs{\proin{\varphi}{h}}^2\right)^{\tfrac{1}{2}}\left(\sum_{h\in\mathcal{H}}\abs{I(h)}^{-2s}\abs{\proin{\varphi}{h}}^2\right)^{\tfrac{1}{2}}\\
&=\sqrt{\frac{\mathcal{Q}^\delta_s(\varphi)}{\gamma_1(s)}}
\sqrt{\frac{\mathcal{E}^\delta_s(\varphi)}{\gamma_2(s)}}.
\end{align*}
So that $\mathcal{Q}^\delta_s(\varphi)\mathcal{E}^\delta_s(\varphi)\geq \gamma(s)\norm{\varphi}^4_2$ as desired.
\end{proof}
\begin{proof}[Proof of Lemma~\ref{lemm:valueEandQinpsi}]
Notice first that (3) follows from (1) and Lemma~\ref{lem:energyvalueH}. In fact, for $\varphi\in\mathcal{S}(\mathcal{H})$,
\begin{align*}
\mathcal{E}^\delta_s(\varphi) &= \iint_{(\mathbb{R}^+)^2}\abs{\frac{\varphi(x)-\varphi(y)}{\delta^s(x,y)}}^2\frac{dx dy}{\delta(x,y)}\\
&=\iint_{(\mathbb{R}^+)^2}\left(\sum_{h\in\mathcal{H}}\proin{\varphi}{h}(h(x)-h(y))\right)^2\frac{dx dy}{\delta^{1+2s}(x,y)}\\
&=\sum_{h\in\mathcal{H}}\proin{\varphi}{h}^2\iint_{(\mathbb{R}^+)^2}\left(\frac{h(x)-h(y)}{\delta^s(x,y)}\right)^2\frac{dx dy}{\delta(x,y)}\\
& \phantom{\sum_{h\in}}+ \sum_{\{(h,\widetilde{h}):\in\mathcal{H}\times\mathcal{H}, h\neq\widetilde{h}\}}\proin{\varphi}{h}\proin{\varphi}{\widetilde{h}}
\iint_{(\mathbb{R}^+)^2}\frac{(h(x)-h(y))}{\delta^s(x,y)}\frac{(\widetilde{h}(x)-\widetilde{h}(y))}{\delta^s(x,y)}\frac{dx dy}{\delta(x,y)}\\
&=\sum_{h\in\mathcal{H}}\proin{\varphi}{h}^2\mathcal{E}^\delta_s(h)+\sum_{\{(h,\widetilde{h}):\in\mathcal{H}\times\mathcal{H}, h\neq\widetilde{h}\}}\proin{\varphi}{h}\proin{\varphi}{\widetilde{h}}E^\delta_s(h,\widetilde{h})\\
&= \gamma_2(s)\sum_{h\in\mathcal{H}}\abs{I(h)}^{-2s}\abs{\proin{\varphi}{h}}.
\end{align*}
The same argument using (b) in Lemma~\ref{lem:VarhQsform} and (2) shows (4). Hence, we only have to prove (1) and (2). Actually (1) is proved in \cite{AcAiBoGo16}, for the sake of completeness we give here a proof in the simple setting of $\mathbb{R}^+$ where we are now working. Let $h$ and $\widetilde{h}$ two different Haar wavelets in $\mathcal{H}$, then
\begin{equation*}
E^\delta_s(h,\widetilde{h})=\iint_{(\mathbb{R}^+)^2}\frac{[h(x)-h(y)][\widetilde{h}(x)-\widetilde{h}(y)]}{\delta^{1+2s}(x,y)} dx dy.
\end{equation*}
Let us consider the two basic cases for the relative positions of $I(h)$ and $\widetilde{I}(h)$. First $I(h)\cap I(\widetilde{h})=\emptyset$ and second $I(h)\varsubsetneqq I(\widetilde{h})$. Assume that $I(h)\cap I(\widetilde{h})=\emptyset$, then $\delta$ is constant on the support of the integrand, so that
\begin{equation*}
E^\delta_s(h,\widetilde{h})=\delta^{-1-2s}(I(h),I(\widetilde{h}))\left\{\iint_{I(h)\times I(\widetilde{h})}h(x)\widetilde{h}(y) dx dy + \iint_{I(\widetilde{h})\times I(h)}(-h(y))\widetilde{h}(x) dx dy \right\}=0,
\end{equation*}
as desired.
Assume now that $I(h)\varsubsetneqq I(\widetilde{h})$. Then the support for the integrand in $E^\delta_s(h,\widetilde{h})$ is contained in the intersection of the supports of $h(x)-h(y)$ and $\widetilde{h}(x)-\widetilde{h}(y)$ which is certainly contained in $[(\mathbb{R}\setminus I(h))\times I(h)]\cup[I(h)\times I(h)]\cup [I(h)\times (\mathbb{R}\setminus I(h))]$. So that
\begin{align*}
E^\delta_s(h,\widetilde{h})&=\int_{x\notin I(h)}\left(\int_{y\in I(h)}\frac{(-h(y))(\widetilde{h}(x)-\widetilde{h}(y))}{\delta^{1+2s}(x,y)} dy\right)dx\\
& + \int_{x\in I(h)}\left(\int_{y\in I(h)}\frac{(h(x)-h(y))(\widetilde{h}(x)-\widetilde{h}(y))}{\delta^{1+2s}(x,y)} dy\right)dx\\
&+ \int_{y\notin I(h)}\left(\int_{x\in I(h)}\frac{h(x)(\widetilde{h}(x)-\widetilde{h}(y))}{\delta^{1+2s}(x,y)} dx\right)dy\\
& = I_1 + I_2 + I_3.
\end{align*}
Let us show that $I_1$, $I_2$ and $I_3$ vanish. For $I_1$ notice that in the inner integral, for fixed $x\notin I(h)=(a(h), b(h)]$ we have two important facts. First $\delta$ is constant on $I(h)$ as a function of $y$ and $\widetilde{h}(y)$ is constant, say $\widetilde{h}(b(h))$, for $y\in I(h)$ because the support of $\widetilde{h}$ is strictly larger than $I(h)$. So that
\begin{equation*}
I_1=-\int_{x\notin I(h)}\frac{1}{\delta(x,I(h))}\left(\int_{y\in I(h)}h(y)(\widetilde{h}(x)-\widetilde{h}(b(h))) dy\right)dx
\end{equation*}
which vanished since the inner integral is zero. $I_3=I_1=0$. For $I_2$ we only have to observe that since $x$ and $y$ belong to $I(h)$ we necessarily have that $\widetilde{h}(x)=\widetilde{h}(y)$, hence $I_2=0$.
Let us finally prove (2), that is $Q^\delta_s(h,\widetilde{h})=0$ for $h\neq\widetilde{h}$ both in $\mathcal{H}$. Let us again consider two cases. Assume first that $I(h)\cap I(\widetilde{h})=\emptyset$. Then
\begin{align*}
Q^\delta_s(h,\widetilde{h}) &= \iint_{(\mathbb{R}^+)^2}\delta^{2s-1}(x,y) h(x)\widetilde{h}(y) dx dy\\
&=\int_{x\in I(h)}h(x)\left(\int_{y\in I(\widetilde{h})}\delta^{2s-1}(x,y)\widetilde{h}(y)dy\right)dx\\
&=\int_{x\in I(h)}h(x)\delta^{2s-1}(x,I(\widetilde{h}))\left(\int_{y\in I(\widetilde{h})}\widetilde{h}(y)dy\right)dx = 0.
\end{align*}
Assume now that $I(h)\varsubsetneqq I(\widetilde{h})$. In this case the support of the integrand in $Q^\delta_s(h,\widetilde{h})$ is contained in $I(h)\times I(\widetilde{h})$
\begin{equation*}
Q^\delta_s(h,\widetilde{h})=\iint_{I(h)\times I(h)}\delta^{2s-1}(x,y) h(x)\widetilde{h}(y) dx dy+ \int_{x\in I(h)}\int_{y\in I(\widetilde{h})\setminus I(h)}\delta^{2s-1}(x,y) h(x)\widetilde{h}(y) dx dy = I_1 + I_2.
\end{equation*}
Let us first deal with $I_2$ that can be written as
\begin{equation*}
I_2=\int_{y\in I(\widetilde{h})\setminus I(h)}\delta^{2s-1}(I(h),y)\widetilde{h}(y)\left(\int_{x\in I(h)} h(x) dx\right)dy=0.
\end{equation*}
Since $I(\widetilde{h})\supsetneqq I(h)$, then $\widetilde{h}$ is constant on $I(h)$ and, with $\widetilde{h}_{|_{I(h)}}$ the restriction of $\widetilde{h}$ to $I(h)$,
\begin{align*}
I_1 &= \iint_{I(h)\times I(h)}\delta^{2s-1}(x,y) h(x)\widetilde{h}(y) dx dy\\
&= \widetilde{h}_{|_{I(h)}}\int_{x\in I(h)} h(x)\left(\int_{y\in I(h)}\delta^{2s-1}(x,y) dy\right)dx\\
&= \widetilde{h}_{|_{I(h)}}\left[\int_{x\in I_-(h)}\frac{1}{\sqrt{\abs{I(h)}}}\left(\int_{y\in I(h)}\delta^{2s-1}(x,y) dy\right) dx - \int_{x\in I_+(h)}\frac{1}{\sqrt{\abs{I(h)}}}\left(\int_{y\in I(h)}\delta^{2s-1}(x,y) dy\right) dx
\right]\\
&= \pm\frac{1}{\sqrt{\abs{I(h)}|I(\tilde{h})|}}\left[\int_{x\in I_-(h)}\left(\int_{y\in I(h)}\delta^{2s-1}(x,y) dy\right)dx - \int_{x\in I_+(h)}\left(\int_{y\in I(h)}\delta^{2s-1}(x,y) dy\right) dx
\right]\\
&= 0.
\end{align*}
The lemma is proved.
\end{proof}

\section{Euclidean fractional uncertainty}\label{sec:euclideanfractionaluncertainty}

From the results in the previous section we aim to prove a fractional uncertainty in terms of the Euclidean energy and position forms. Following closely the notation of Section~\ref{sec:dyadicfractionaluncertainty} we define the energy bilinear form, the energy quadratic form, the position bilinear form, and the position quadratic form for $0<s<\tfrac{1}{2}$,
\begin{eqnarray*}
  E_s(\varphi,\psi) &=& \iint_{\mathbb{R}^2}\frac{[\varphi(x)-\varphi(y)]}{\abs{x-y}^s}\frac{[\overline{\psi}(x)-\overline{\psi}(y)]}{\abs{x-y}^s} \frac{dx dy}{\abs{x-y}}; \\
  Q_s(\varphi,\psi) &=& \iint_{\mathbb{R}^2}[\abs{x-y}^s\varphi(x)][\abs{x-y}^s\overline{\psi}(y)] \frac{dx dy}{\abs{x-y}}; \\
  \mathcal{E}_s(\varphi) &=& \iint_{\mathbb{R}^2}\frac{\abs{\varphi(x)-\varphi(y)}^2}{\abs{x-y}^{2s}}\frac{dx dy}{\abs{x-y}} \textrm{ and} \\
  \mathcal{Q}_s(\varphi) &=& \iint_{\mathbb{R}^2}\abs{x-y}^{2s}\varphi(x)\overline{\varphi}(y)\frac{dx dy}{\abs{x-y}}.
\end{eqnarray*}
With these notation, we are in position to state the main result of this section.
\begin{theorem}
Let $0<s<\tfrac{1}{2}$ and $\gamma(s)$ be the constant in Theorem~\ref{thm:inequalityQandE}. Then, the inequality
\begin{equation*}
\mathcal{Q}_s(\abs{\varphi})\cdot\mathcal{E}_s(\varphi)\geq \gamma(s)
\end{equation*}
holds for every $\varphi\in L^2(\mathbb{R})$ with $\norm{\varphi}_2=1$.
\end{theorem}
\begin{proof}
Let $\varepsilon>0$ be given. Since $\int_{\mathbb{R}}\abs{\varphi(x)}^2dx=1$, then there exists $x_0\in \mathbb{R}$ such that $\int_{x_0}^{\infty}\abs{\varphi(x)}^2dx\geq 1-\varepsilon$. Let $\mathcal{D}_{x_0}=x_0+\mathcal{D}=\{J=x_0+I: I\in\mathcal{D}\}$ and $\delta_{x_0}(x,y)=\inf \{\abs{J}: x, y\in J, J\in \mathcal{D}_{x_0}\}$. Observe also that for $x$ and $y$ both larger than $x_0$ we have that $\abs{x-y}\leq\delta_{x_0}(x,y)$. With this change of the origin all the quantities in Section~\ref{sec:dyadicfractionaluncertainty} are well defined and in particular Theorem~\ref{thm:inequalityQandE} holds in the half line $x>x_0$. If $\varphi^{x_0}$ denotes the restriction of $\varphi$ to $(x_0,\infty)$, we have
\begin{equation*}
\mathcal{Q}^{\delta_{x_0}}_{s}(\varphi^{x_0})\cdot\mathcal{E}^{\delta_{x_0}}_{s}(\varphi^{x_0})\geq \gamma(s)\norm{\varphi^{x_0}}^4_2\geq \gamma(s)(1-\varepsilon)^2.
\end{equation*}
More explicitly
\begin{equation*}
\left(\iint_{x>x_0, y>x_0}\delta_{x_0}^{2s}(x,y)\varphi^{x_0}(x)\overline{\varphi^{x_0}}(y)\frac{dx dy}{\delta_{x_0}(x,y)}\right) \left(\iint_{x>x_0, y>x_0}\frac{\abs{\varphi^{x_0}(x)-\overline{\varphi^{x_0}}(y)}^2}{\delta_{x_0}^{2s}(x,y)}\frac{dx dy}{\delta_{x_0}(x,y)}\right)\geq \gamma(s)(1-\varepsilon)^2.
\end{equation*}
Since $0<s<\tfrac{1}{2}$, the first factor on the left is bounded above by
\begin{equation*}
\iint_{x>x_0, y>x_0}\frac{\abs{\varphi^{x_0}(x)}\abs{\overline{\varphi^{x_0}}(y)}}{\abs{x-y}^{1-2s}}dx dy\leq\mathcal{Q}_s(\abs{\varphi}).
\end{equation*}
For the second factor we also have the upper bound
\begin{equation*}
\iint_{x>x_0, y>x_0}\frac{\abs{\varphi(x)-\overline{\varphi}(y)}^2}{\abs{x-y}^{2s}}\frac{dx dy}{\abs{x-y}}\leq\mathcal{E}_s(\varphi).
\end{equation*}
Hence, for every $\varepsilon>0$, we have
\begin{equation*}
\mathcal{Q}_s(\abs{\varphi})\cdot\mathcal{E}_s(\varphi)
\geq \gamma(s)(1-\varepsilon)^2,
\end{equation*}
and we are done.
\end{proof}

From Lemma~\ref{lem:VarhQsform}~(b) and Lemma~\ref{lemm:valueEandQinpsi} we see that for each $h\in\mathcal{H}$ that $\mathcal{Q}_s^\delta(h)\mathcal{E}^\delta_s(h)=\gamma(s)$. For the Euclidean case, we have the following result.

\begin{proposition}
For every $h\in\mathcal{H}$ we have and $0<s<\tfrac{1}{2}$, $\mathcal{Q}_s(\abs{h})\mathcal{E}_s(h)=\frac{6}{s^2(1-4s^2)}$, in fact
\begin{enumerate}[(a)]
\item  $\mathcal{Q}_s(\abs{h})=\frac{1}{s(2s+1)}\abs{I(h)}^{2s}$, and
\item  $\mathcal{E}_s(h)=\frac{6}{s(1-2s)}\abs{I(h)}^{-2s}.$
\end{enumerate}

\end{proposition}
\begin{proof}[Proof of (a)]
	
	\begin{align*}
	\mathcal{Q}_s(\abs{h}) & = \iint_{\mathbb{R}^2}\abs{x-y}^{2s-1}\abs{h(x)}\abs{h(y)} dx dy\\
	&=\frac{1}{\abs{I(h)}}\left(\iint_{I_+(h)\times I_+(h)} + \iint_{I_-(h)\times I_-(h)} + \iint_{I_-(h)\times I_+(h)} + \iint_{I_+(h)\times I_-(h)}\right)\\
	&= \frac{1}{\abs{I(h)}}(A_1 + A_2 + A_3 + A_4).
	\end{align*}
We denote $I_+(h)=I_+=(a_+,b_+]$ and $I_-(h)=I_-=(a_-,b_-]$, notice that $b_-=a_+$, then
\begin{align*}
A_1 &= \int_{x\in I_+}\left(\int_{y\in I_+}\abs{x-y}^{2s-1} dy\right) dx\\
&= \frac{1}{2s}\int_{a_+}^{b_+}\left((x-a_+)^{2s}+(b_+-x)^{2s}\right)dx\\
&= \frac{1}{2s}\frac{1}{2s+1}\left((b_+-a_+)^{2s+1}+(b_+-a_+)^{2s+1}\right)\\
&= \frac{2}{(2s+1)2s}\abs{I_+}^{2s+1}\\
&= \frac{2^{-2s}}{(2s+1)2s}\abs{I}^{2s+1}.
\end{align*}
Notice that $A_1=A_2$ and $A_3=A_4$. Now
\begin{align*}
A_3 &= \int_{x\in I_-}\left(\int_{y\in I_+}\abs{x-y}^{2s-1} dy\right) dx\\
&= \int_{x\in I_-}\left(\int_{y\in I_+}(y-x)^{2s-1} dy\right) dx\\
&= \frac{1}{2s(2s+1)}\left((b_+-a_-)^{2s+1}-(b_+-b_-)^{2s+1}-(a_+-a_-)^{2s+1}+(a_+-b_-)^{2s+1}\right)\\
&= \frac{1}{2s(2s+1)}\left(\abs{I}^{2s+1}-2\left(\frac{\abs{I}}{2}\right)^{2s+1}\right)\\
&= \frac{1}{2s(2s+1)}\left(1-2^{-2s}\right)\abs{I}^{2s+1}.
\end{align*}
Hence
\begin{equation*}
\mathcal{Q}_s(\abs{h}) = \frac{2}{2s(2s+1)}(2^{-2s}+1-2^{-2s})\abs{I}^{2s}
= \frac{1}{s(2s+1)}\abs{I}^{2s}.
\end{equation*}

\noindent\textit{Proof of (b).}

\begin{equation*}
\mathcal{E}_s(h) = \iint_{\mathbb{R}^2}\frac{\abs{h(x)-h(y)}^2}{\abs{x-y}^{1+2s}} dx dy\\
= \iint_{I\times I} + \iint_{(\mathbb{R}\setminus I)\times I} + \iint_{I\times (\mathbb{R}\setminus I)}\\
= A + B + C.
\end{equation*}
For $A$ we have that, $A = \iint_{I_-\times I_+} + \iint_{I_+\times I_-} = 2\iint_{I_-\times I_+}$, because $\abs{h(x)-h(y)}$ vanishes on $(I_-\times I_-)\cup(I_+\times I_+)$. So that
\begin{align*}
A &= \frac{8}{\abs{I}}\int_{a_-}^{b_-}\left(\int_{a_+}^{b_+}\frac{dy}{(y-x)^{1+2s}}\right) dx\\
&= \frac{4}{s(1-2s)}\frac{1}{\abs{I}}[(a_+-a_-)^{1-2s} +(b_+-b_-)^{1-2s} - (b_+-a_-)^{1-2s}]\\
&= \frac{4}{s(1-2s)}\abs{I}^{-2s}.
\end{align*}
Notice now that $B + C = 4 B_1$, with
\begin{equation*}
B_1 =\frac{1}{\abs{I}}\int_{a}^b\left(\int_{-\infty}^{a}\frac{dx}{(y-x)^{1+2s}}\right)dy= \frac{1}{\abs{I}}\int_{a}^b\left(\frac{1}{2s}(y-x)^{-2s}\right)dy
= \frac{1}{2s(1-2s)}\abs{I}^{-2s}.
\end{equation*}
Hence $\mathcal{E}_s(h)=\frac{6}{s(1-2s)}\abs{I}^{-2s}$.
\end{proof}

Let us finally point out that $\gamma(s)$ tends to zero when $s$ tends to $\tfrac{1}{2}$. But it tends to infinity for $s\to 0$.



\bigskip

\bigskip
\noindent{\footnotesize

\noindent Hugo Aimar and Ivana G\'{o}mez. \textsc{Instituto de Matem\'{a}tica Aplicada del Litoral, UNL, CONICET, FIQ.}

\smallskip
\noindent Pablo Bolcatto. \textsc{Instituto de Matem\'{a}tica Aplicada del Litoral, UNL, CONICET, FHUC.}	

\smallskip
\noindent\textmd{IMAL, CCT CONICET Santa Fe, Predio ``Alberto Cassano'', Colectora Ruta Nac.~168 km 0, Paraje El Pozo, S3007ABA Santa Fe, Argentina.}
}
\bigskip

\end{document}